\documentclass[12pt]{amsart}
\usepackage{amssymb}
\input amssym.def
\usepackage{pstricks}
\usepackage{url}
\usepackage[all]{xy}
\newtheorem{theorem}{Theorem}
\newtheorem{lemma}[theorem]{Lemma}
\newtheorem{prop}[theorem]{Proposition}
\newtheorem{cor}[theorem]{Corollary}
\theoremstyle{definition}
\newtheorem{rem}[theorem]{Remark}

\newcommand{\M}{\mathcal{M}}

\newcommand{\Z}{\mathbb{Z}}
\newcommand{\R}{\mathbb{R}}
\newcommand{\C}{\mathbb{C}}
\newcommand{\Dif}{\textrm{Diff}}
\newcommand{\ab}{\mathrm{ab}}
\newcommand{\im}{\mathrm{Im}}
\newcommand{\diag}{\mathrm{diag}}
\newcommand{\GL}{\mathrm{GL}}
\newcommand{\Sp}{\mathrm{Sp}}
\newcommand{\lr}[1]{\left<#1\right>}
\numberwithin{equation}{section}
\numberwithin{theorem}{section}
\author{B\l{}a\.zej Szepietowski}
\title[Linear representations of the mapping class group]{Low dimensional linear representations of the mapping class group of a nonorientable surface}
\address[]{Institute of Mathematics, Gda\'nsk University, Wita Stwosza 57,
80-952 Gda\'nsk, Poland} 
\email{blaszep@mat.ug.egu.pl}
\thanks{Supported by NCN grant nr 2012/05/B/ST1/02171.}

\begin{document}
\maketitle
\begin{abstract}
Suppose that $f$ is a homomorphism from the mapping class group $\M(N_{g,n})$ of a nonorientable surface of genus $g$ with $n$ boundary components, to $\GL(m,\C)$.
We prove that if $g\ge 5$, $n\le 1$ and $m\le g-2$, then $f$ factors through the abelianization of $\M(N_{g,n})$, which is $\Z_2\times\Z_2$ for $g\in\{5,6\}$ and $\Z_2$ for $g\ge 7$. If $g\ge 7$, $n=0$ and $m=g-1$, then either $f$ has finite image (of order at most two if $g\ne 8$), or it is conjugate to one of four ``homological representations''. As an application we prove that for $g\ge 5$ and $h<g$, every homomorphism $\M(N_{g,0})\to\M(N_{h,0})$ factors through the abelianization of $\M(N_{g,0})$. 
\end{abstract}
\section{Introduction}
For a compact surface $F$, its {\it mapping class group} $\M(F)$ is the group of isotopy classes of all, orientation preserving if $F$ is orientable, homeomorphisms $F\to F$ equal to the identity on the boundary of $F$. A compact surface of genus $g$ with $n$ boundary components will be denoted by $S_{g,n}$ if it is orientable, or by $N_{g,n}$ if it is nonorientable. If $n=0$ then we drop it in the notation and write simply $S_g$ or $N_g$. The first integral homology group of $F$ will be denoted by $H_1(F)$.

After fixing a basis of $H_1(S_g)$, the action of $\M(S_g)$ on $H_1(S_g)$ gives rise to a homomorphism
$\M(S_g)\to\Sp(2g,\Z)$, which is well known to be surjective, and whose kernel is known as the Torelli group. Gluing a disc along each boundary component of $S_{g,n}$ induces an epimorphism $\M(S_{g,n})\to\M(S_g)$, and by composing it with $\M(S_g)\to\Sp(2g,\Z)$, and then with the inclusion $\Sp(2g,\Z)\hookrightarrow\GL(2g,\C)$ we obtain the map
$\Phi\colon\M(S_{g,n})\to\GL(2g,\C)$. Recently, the following two results were proved by J. Franks, M. Handel and M. Korkmaz.
\begin{theorem}[\cite{FH,KorkRep}]\label{FHK}
Let $g\ge 2$, $m\le 2g-1$ and let $f\colon\M(S_{g,n})\to\GL(m,\C)$ be a homomorphism. Then $f$ is trivial if $g\ge 3$, and $\im(f)$  is a quotient of $\Z_{10}$ if $g=2$.
\end{theorem} 
We say that two homomorphism $f_1$, $f_2$ from a group $G$ to a group $H$ are {\it conjugate} if there exits $h\in H$ such that $f_2(x)=hf_1(x)h^{-1}$ for $x\in G$.
\begin{theorem}[\cite{KorkSymp}]\label{KorU}
For $g\ge 3$, every nontrivial homomorphism $f\colon\M(S_{g,n})\to\GL(2g,\C)$ is  conjugate to the map $\Phi$.
\end{theorem} 
In this paper we prove analogous results for $\M(N_g)$. Fix $g\ge 3$. 
Let $R_g$ denote the quotient of $H_1(N_g)$ by its torsion. Hence, $R_g$ is a free $\Z$-module of rank $g-1$.
There is covering $P\colon S_{g-1}\to N_g$ of degree two. By a theorem of Birman and Chillingworth \cite{BC}, $\M(N_g)$ is isomorphic to the subgroup of $\M(S_{g-1})$ consisting of the isotopy classes of orientation preserving lifts of homeomorphisms of $N_g$, which gives an action of $N_g$ on $H_1(S_{g-1})$.
Let $K_g\subset H_1(S_{g-1})$ be the kernel of the composition of the induced map
$P_\ast\colon H_1(S_{g-1})\to H_1(N_g)$ with the canonical projection  $H_1(N_g)\to R_g$. Then $K_g$ is $\M(N_g)$-invariant subgroup of rank $g-1$ and 
 we have two homomorphisms
\[\Psi_1\colon\M(N_g)\to\GL(K_g)\quad\textrm{and}\quad \Psi_2\colon\M(N_g)\to\GL(H_1(S_{g-1})/K_g),\] which after fixing bases will be treated as representations of $\M(N_g)$ in $\GL(g-1,\C)$. We will see that these representations are not conjugate, although $\ker\Psi_1=\ker\Psi_2$.

Our first result is the following.
\begin{theorem}\label{MNtoGLfact}
Suppose that $n\le 1$, $g\ge 5$, $m\le g-2$ and $f\colon\M(N_{g,n})\to\GL(m,\C)$ is a nontrivial homomorphism. Then $\im(f)$ is ether $\Z_2$ or $\Z_2\times\Z_2$,   the latter case being possible only for $g=5$ or $6$.
\end{theorem}
Theorem \ref{MNtoGLfact} was proved in \cite{KorkRep}, in a more general setting of punctured surfaces, under additional assumption that $m\le g-3$ if $g$ is even. Therefore, the only novelty of our result is that it also covers the case $m=g-2$ for even $g$. 
As an application of Theorem \ref{MNtoGLfact} we prove the following result, which solves Problem 3.3 in \cite{KorkProb}.
\begin{theorem}\label{MNgtoMNh}
Suppose that $g\ge 5$, $h<g$ and $f\colon\M(N_g)\to\M(N_h)$ is a nontrivial homomorphism. Then $\im(f)$ is as in Theorem \ref{MNtoGLfact}.
\end{theorem}
Analogous theorem for mapping class groups of orientable surfaces was proved in \cite{HK}, see also \cite{AS}. We will prove that both Theorem \ref{MNtoGLfact} and Theorem \ref{MNgtoMNh} fail for $g=4$, by showing that there is a homomorphism from $\M(N_4)$ to $\M(N_3)\cong\GL(2,\Z)$, whose image is isomorphic to the infinite dihedral group.

Suppose that $g\ge 7$. Then the abelianization of $\M(N_g)$ is $\Z_2$ and we denote by $\ab\colon\M(N_g)\to\Z_2$ the canonical projection. For $i=1,2$ we set $\Psi'_i=(-1)^\ab\Psi_i$. Our next result is the following.
\begin{theorem}\label{MNtoGLg-1}
Suppose that $g\ge 7$, $g\ne 8$ and $f\colon\M(N_g)\to\GL(g-1,\C)$ is a nontrivial homomorphism. Then either $\im(f)\cong\Z_2$, or $f$ is conjugate to one of  $\Psi_1$, $\Psi'_1$, $\Psi_2$, $\Psi'_2$.
\end{theorem}
For $g=8$ other representations of $\M(N_8)$ in $\GL(7,\C)$ occur, related to the fact that there is an epimorphism $\epsilon\colon\M(N_8)\to\Sp(6,\Z_2)$ and the last group admits irreducible representations in $\GL(7,\C)$ (see \cite{Atlas}). We prove the following result.
\begin{theorem}\label{MN8toGL7}
Suppose that $f\colon\M(N_8)\to\GL(7,\C)$ is a homomorphism. Then one of the following holds.
\begin{itemize}
\item[(1)] $\im(f)\cong\Z_2$.
\item[(2)] $f$ or $(-1)^\ab f$ factors through $\epsilon\colon\M(N_8)\to\Sp(6,\Z_2)$.
\item[(3)] $f$ is conjugate to one of $\Psi_1$, $\Psi'_1$, $\Psi_2$, $\Psi'_2$. 
\end{itemize}
\end{theorem}
To prove our theorems we use the ideas and results from \cite{FH,KorkRep,KorkSymp} with necessary modifications. While the case of odd genus is relatively easy, the case of even genus requires much more effort. This phenomenon is typical for the mapping group of a nonorientable surface.

Throughout this paper we will often have to solve an equation of the form $L=R$, where $L$ and $R$ are products of matrices from $\GL(m,\C)$ with some unknown coefficients. Although the dimension $m$ is variable, the calculations of $L$ and $R$ always reduce to multiplication of blocks of size at most $7\times 7$. With some patience, such calculations could be done by hand, but it is definitely easier to use a computer. We used GAP, but of course, any program that performs symbolic operations on matrices, could be used as well.
\section{Notation and algebraic preliminaries}\label{nota}
Suppose that $m\ge 2$ is fixed. We denote by $I_m$ the identity matrix of dimension $m$. We will sometimes write simply $I$, if $m$ is clear from the context. We denote by $E_{ij}$ the elementary matrix with $1$ on the position $(i,j)$ and $0$ elsewhere. Suppose that $M_1,\dots,M_k$ are nonsingular square matrices of dimensions $m_1,\dots,m_k$, where $m_1+\cdots+m_k=m$. Then we denote by $\diag\left(M_1,\dots,M_k\right)$ the $m\times m$ matrix with $M_1,\dots,M_k$ on the main diagonal and zeros elsewhere.
Set
\[
V=\begin{pmatrix}1&1\\0&1\end{pmatrix},\quad \widehat{V}=\begin{pmatrix}1&0\\-1&1\end{pmatrix},\quad
W=\begin{pmatrix}1&1&0&-1\\0&1&0&0\\0&-1&1&1\\0&0&0&1\end{pmatrix}\]
For $2\le 2i\le m$ we define
\[A_i=\diag\left(I_{2i-2},V,I_{m-2i}\right),\quad 
B_i=\diag\left(I_{2i-2},\widehat{V},I_{m-2i}\right),\] 
and for $2\le 2j\le m-2$,
\[C_j=\diag\left(I_{2j-2},W,I_{m-2-2j}\right).\]
The proof of the following lemma is straightforward and we leave it as an exercise  (c.f. \cite[Lemma 2.2]{KorkSymp}).
\begin{lemma}\label{diag}
Suppose that $1\le k\le l\le m/2$ and $M\in\GL(m,\C)$ satisfies $A_iM=MA_i$,  $B_iM=MB_i$ and $C_jM=MC_j$ for all $i,j$ such that $k\le i\le l$, $k\le j\le l-1$.  Then $M$ has the form
\begin{equation*}\label{stars}
\begin{pmatrix}\ast&0&\ast\\
0&\lambda I_{2(l-k+1)}&0\\
\ast&0&\ast\end{pmatrix},
\end{equation*} for some $\lambda\in\C^\ast$, where the top-left $\lambda$ of the block $\lambda I_{2(l-k+1)}$  is at the position $(2k-1,2k-1)$. 
\end{lemma}

Suppose that $L\in\GL(m,\C)$ and $\lambda$ is an eigenvalue of $L$. Then we denote by $\#\lambda$ the multiplicity of $\lambda$. For $k\ge 1$ we denote by $E^k(L,\lambda)$ the space $\ker(E-\lambda I)^k$. Thus
$E^1(L,\lambda)$ is the eigenspace of $L$ with respect to $\lambda$, and it will be also denoted by $E(L,\lambda)$. Note that if $L'\in\GL(m,\C)$ commutes with $L$, then the spaces $E^k(L,\lambda)$ are $L'$-invariant for $k\ge 1$.

For $k\ge 2$ we denote by $\mathfrak{S}_k$  the full symmetric group of the set
$\{1,\dots,k\}$. It is generated by the transpositions $\sigma_i=(i,i+1)$ for
$1\le i\le k-1$. We will need the following result from the representation theory of the symmetric group, see for example \cite[Exercise 4.14]{FulHar}.
\begin{lemma}\label{repsym}
For $k\ge 5$, $\mathfrak{S}_k$ has no irreducible  representation (over $\C$) of  dimension $1<m<k-1$. If $k\ge 7$, then $\mathfrak{S}_k$ has two irreducible representations of dimension $k-1$: the standard one and the tensor product of the standard and sign representations. 
\end{lemma}
\section{Mapping class group of a nonorientable surface}
Let $n\in\{0,1\}$ and $g\ge 2$.
Let us  represent $N_{g,n}$ as a sphere (if $n=0$) or a disc (if $n=1$) with $g$ crosscaps. This means that interiors of $g$ small pairwise disjoint discs should be removed from the sphere/disc, and then antipodal points in each of the resulting boundary components should be identified. Let us arrange the crosscaps as shown on Figure \ref{xiI} and number them from $1$ to $g$. 
For each nonempty subset $I\subseteq\{1,\dots,g\}$ let $\xi_I$ be the simple closed curve shown on Figure \ref{xiI}.  Note that $\xi_I$ is two-sided if and only if $I$ has even number of elements. In such case $t_{\xi_I}$ will be the Dehn twist about $\gamma_I$ in the direction indicated by arrows on Figure \ref{xiI}. 
\begin{figure}
\input{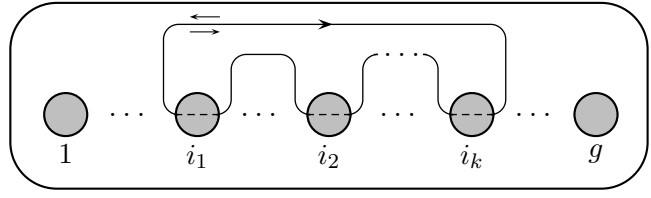}
\caption{\label{xiI} The surface $N_{g,n}$ and the curve $\xi_I$ for $I=\{i_1,i_2,\dots,i_k\}$.}
\end{figure}
We will write $\xi_i$ instead of $\xi_{\{i\}}$.
The following curves will play a special role and so we give them different names.
\begin{itemize}
\item $\delta_i=\xi_{\{i,i+1\}}$ for $1\le i\le g-1$,
\item $\varepsilon_j=\xi_{\{1,2,\dots,2j\}}$ for $2\le 2j\le g$.
\end{itemize}
Note that $\varepsilon_1=\delta_1$.

For $1\le i\le g-1$  
we define the {\it crosscap transposition} $u_i$ to be the isotopy class of the homeomorphism interchanging the $i$'th and the $(i+1)$'st crosscaps as shown on Figure \ref{U}, and equal to the identity outside a disc containing these crosscaps.

\begin{figure}
\input{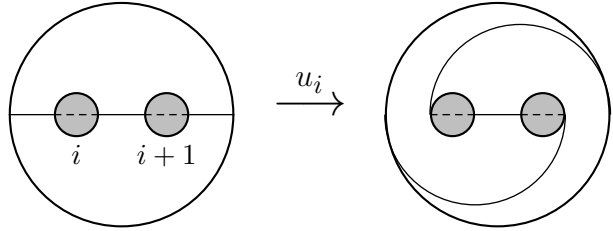}
\caption{\label{U}The crosscap transposition $u_i$.}
\end{figure}
The groups $\M(N_{1,n})$ are trivial for $n\le 1$ by \cite[Theorem 3.4]{E}, we have $\M(N_2)\cong\Z_2\times\Z_2$ by \cite{Lick}, and it follows from \cite{BC} that $\M(N_3)\cong\GL(2,\Z)$. For $g\ge 3$, a finite generating set for $\M(N_{g,n})$ was given in \cite{Chill} for $n=0$ and \cite{Stu_bdr} for $n>0$. For $n\le 1$ this set can be reduced to the one given in the 
following theorem, which can be deduced form the main result of \cite{PSz}.
\begin{theorem}\label{gener}
For $g\ge 4$ and $n\in\{0,1\}$, $\M(N_{g,n})$ is generated by
$u_{g-1}$, $t_{\varepsilon_2}$ and $t_{\delta_i}$ for $1\le i\le g-1$.
\end{theorem}
If $n>1$, then we consider $N_{g,n}$ as the result of gluing $S_{0,n+1}$ to  $N_{g,1}$ along the boundary component.
We will need the following relations, satisfied in $\M(N_{g,n})$. Those between Dehn twists are the well know disjointness and braid relations. 
\begin{itemize}
\item[(R1)] $t_{\delta_i}t_{\delta_j}=t_{\delta_j}t_{\delta_i}\quad$ for $|i-j|>1$,
\item[(R2)] $t_{\varepsilon_i}t_{\varepsilon_j}=t_{\varepsilon_j}t_{\varepsilon_i}\quad$ for all $i,j$,
\item[(R3)] $t_{\varepsilon_i}t_{\delta_j}=t_{\delta_j}t_{\varepsilon_i}\quad$ for  $j\ne 2i$,
\item[(R4)] $t_{\delta_i}t_{\delta_{i+1}}t_{\delta_i}=t_{\delta_{i+1}}t_{\delta_i}t_{\delta_{i+1}}\quad$ for $1\le i\le g-2$,
\item[(R5)] $t_{\varepsilon_i}t_{\delta_{2i}}t_{\varepsilon_i}=t_{\delta_{2i}}t_{\varepsilon_i}t_{\delta_{2i}}$ 
 for $2i<g$;
\end{itemize}
The relations involving crosscap transpositions are not so well known and we refer the reader to \cite{PSz} and \cite{SzepB} for their proofs.
\begin{itemize}
\item[(R6)] $t_{\delta_i}u_j=u_jt_{\delta_i}\quad$ for $|i-j|>1$,
\item[(R7)] $u_iu_j=u_ju_i\quad$ for $|i-j|>1$,
\item[(R8)] $t_{\varepsilon_i}u_j=u_jt_{\varepsilon_i}\quad$ for $j>2i$,
\item[(R9)] $u_iu_{i+1}u_i=u_{i+1}u_iu_{i+1}\quad$ for $1\le i\le g-2$,
\item[(R10)] $t_{\delta_i}u_{i+1}u_i=u_{i+1}u_it_{\delta_{i+1}}\quad$ for $1\le i\le g-2$,
\item[(R11)] $u_{i+1}t_{\delta_i}t_{\delta_{i+1}}u_i=t_{\delta_i}t_{\delta_{i+1}}\quad$ for $1\le i\le g-2$;
\item[(R12)] $t_{\delta_i}u_it_{\delta_i}=u_i\quad$ for $1\le i\le g-1$.
\end{itemize}
If follows from (R4) that all $t_{\delta_i}$ are conjugate for $1\le i\le g-1$, by (R5) $t_{\varepsilon_j}$ is conjugate to 
$t_{\delta_{2j}}$ for $2j<g$, and by (R12) $t_{\delta_i}$ is conjugate to $t_{\delta_i}^{-1}$. Similarly, by (R9) all $u_i$ are conjugate for $1\le i\le g-1$, and by (R11) $u_i$ is conjugate to $u_i^{-1}$.

For a group $G$ we denote the abelianization $G/[G,G]$ by $G^\ab$.
The following theorem is proved in \cite{KorkH1} for $n=0$ and generalised to $n>0$ in \cite{Stu_bdr}.
\begin{theorem}\label{abNg}
For $n\le 1$ and $g\ge 3$,  $\M(N_{g,n})^\ab$ has the following presentation as a $\Z$-module.
\begin{align*}
&\lr{[t_{\delta_1}], [t_{\varepsilon_2}], [u_1]\,|\,2[t_{\delta_1}]= 2[t_{\varepsilon_2}]=2[u_1]=0}\quad\textrm{if }g=4,\\
&\lr{[t_{\delta_1}], [u_1]\,|\,2[t_{\delta_1}]=2[u_1]=0}\quad\textrm{if  }g\in\{3,5,6\},\\
&\lr{[u_1]\,|\,2[u_1]=0}\quad\textrm{if }g\ge 7.
\end{align*}
In particular, for $g\ge 7$ we have $[t_{\delta_1}]=0$.
\end{theorem}
\begin{lemma}\label{com_norm}
For $g\ge 5$ and $n\le 1$ let $\alpha$, $\beta$ be two-sided curves on $N_{g,n}$, intersecting transversally in one point.
If $f\colon\M(N_{g,n})\to G$ is a homomorphism, such that $f(t_\alpha)$ commutes with $f(t_\beta)$, then $\im(f)$ is abelian.
\end{lemma}
\begin{proof}
Let $N=N_{g,n}$ and  $\M=\M(N_{g,n})$. Fix a regular neighbourhood $A$ of $\alpha\cup\beta$. Note that $A$ is homeomorphic to $S_{1,1}$ and $N\backslash A$ is homeomorphic to $N_{g-2,1}$. It follows that for each $i\le g-2$ there is a homeomorphism $h\colon N\to N$ such that $h(\alpha)=\delta_i$ and $h(\beta)=\delta_{i+1}$. 
It follows that $ht_\alpha h^{-1}=t^{\varepsilon_1}_{\delta_i}$ and $ht_\beta h^{-1}=t^{\varepsilon_2}_{\delta_{i+1}}$, where $\varepsilon_j\in\{-1,1\}$ for $j=1,2$. Hence $f(t_{\delta_i})$ commutes with $f(t_{\delta_{i+1}})$ and by the braid relation (R4) $f(t_{\delta_i})=f(t_{\delta_{i+1}})$. Analogously,
$f(t_{\varepsilon_2})=f(t_{\delta_4})$.
By Theorem \ref{gener}, $\im(f)$ is generated by $f(t_{\delta_1})$ and $f(u_{g-1})$, and since $u_{g-1}$ commutes with $t_{\delta_1}$, thus $\im(f)$ is abelian. 
 \end{proof}
\begin{lemma}\label{tsq}
Suppose that $g\ge 4$ and $f\colon\M(N_{g,n})\to G$ is a homomorphism.
If $f(t_{\varepsilon_i})=f(t_{\delta_{j}})$ for some $2i+1\le j\le g-1$, then $f(t^2_{\delta_1})=1$. 
\end{lemma}
\begin{proof}
Set $x=f(t_{\varepsilon_i})=f(t_{\delta_j})$ and $y=f(u_j)$. By the relation (R8) we have $xy=yx$, and by (R12) $xyx=y$. Hence $x^2=1$ which finishes the proof, because $t_{\delta_j}$ is conjugate to $t_{\delta_1}$.
\end{proof}
\begin{figure}
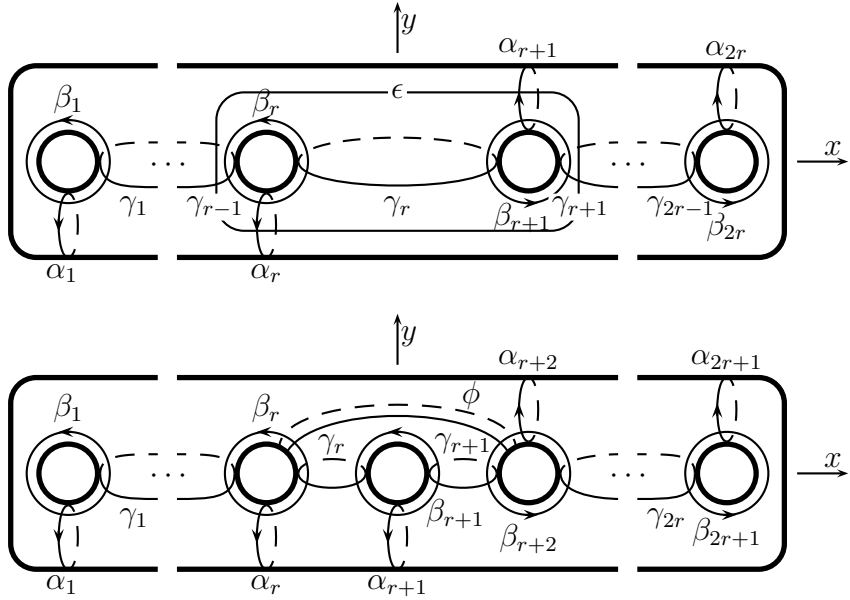

\begin{tabular}{l}
\input{fig3}\\
\input{fig4}
\end{tabular}
\caption{\label{S2r}The surface $S_{g-1}$ for $g=2r+1$ (top) and $g=2r+2$ (bottom).}
\end{figure}
Let $g=2r+s$, where $r\ge 1$, $s\in\{1,2\}$ and $S=S_{g-1}$.
Consider $S$ as being embedded in $\R^3$ in such a way that it is invariant under the reflections about the $xy$, $xz$ and $yz$ planes, as shown on Figure \ref{S2r}. We define a homeomorphism $j\colon S\to S$ as $j(x,y,z)=(-x,-y,-z)$. The quotient space $S/j$ is a nonorientable surface of genus $g$ and the projection $p\colon S\to S/j$ is a covering map of degree $2$.
Let $S'$ be the subsurface of $S$ consisting of points $(x,y,z)\in S$ with
$x\le-\varepsilon$, where $\varepsilon$ is a positive constant, so small that
$S'$ is homeomorphic to $S_{r,s}$. If $g$ is even, then one of the boundary components of $S'$ is isotopic to $\alpha_{r+1}$. In this paper we identify isotopic curves, and therefore we will treat $\alpha_{r+1}$ as a curve on $S'$.
 Note that the restriction of $p$ to $S'$ is an embedding. For odd $g$ we define $\gamma'$ to be the arc of $\gamma_r$ consisting of points with $x\le 0$. For even $g$ we define $\beta'$ to be the arc of $\beta_{r+1}$ consisting of points with $x\le 0$. Note that $p(\gamma')$ and $p(\beta')$ are one-sided simple closed curves on $S/j$.

\begin{prop}\label{HomeoCov}
There is a homeomorphism $\varphi\colon S_{g-1}/j\to N_{g}$ such that, for $P=\varphi\circ p$, up to isotopy
\begin{itemize}
\item[(1)] $P(\beta_i)=\delta_{2i}$ for $1\le i\le r$, 
\item[(2)] $P(\alpha_i)=\varepsilon_i$ for $2\le 2i\le g$, 
\item[(3)] $P(\gamma_i)=\delta_{2i+1}$ for $2\le 2i\le g-2$, 
\item[(4)] $P(\gamma')=\xi_g$ if $g$ is odd,
\item[(5)] $P(\beta')=\xi_g$ if $g$ is even.
\end{itemize} 
\end{prop}
\begin{proof}
Observe that the curves $\delta_i$ for $1\le i\le g-1$ form a chain of two-sided curves, which means that $\delta_i$ and $\delta_j$ intersect at one point if $|i-j|=1$, and they are disjoint otherwise. It follows that a regular neighbourhood of the union of $\delta_i$ for $1\le i\le g-1$ is homeomorphic to $S_{r,s}$. Let $\Sigma$ be such a neighbourhood, which may be taken to contain the curves $\varepsilon_i$ for $2\le 2i\le g$ (if $g$ is even, then one of the boundary components of $\Sigma$ is isotopic to $\varepsilon_{r+1}$). Note that $\varepsilon_i$, $\varepsilon_{i+1}$ and $\delta_{2i+1}$ bound a pair of pants for
$2\le 2i\le g-2$. It follows that there exists a homeomorphism $\varphi\colon S_{g-1}/j\to N_{g}$ such that, for $P=\varphi\circ p$, we have $P(S')=\Sigma$ and the conditions (1, 2, 3) are satisfied. Observe that $N_g\backslash\Sigma$ is a M\"obius strip (if $g$ is odd) or an annulus (if $g$ is even), whose core (isotopic to $\xi_{\{1,\dots,g\}}$) intersects $\xi_g$ once. By looking at the intersection of $\xi_g$ with the curves $\delta_i$, $\varepsilon_j$ it is easy to see that $\varphi$ can be taken to satisfy also the condition (4) or (5).
\end{proof}
\begin{cor}\label{HomSN}
There is a homomorphism $\iota\colon\M(S')\to\M(N_{g,n})$ such that
\begin{itemize}
\item $\iota(t_{\beta_i})=t_{\delta_{2i}}$ for $1\le i\le r$, 
\item $\iota(t_{\alpha_i})=t_{\varepsilon_i}$ for $2\le 2i\le g$, 
\item $\iota(t_{\gamma_i})=t_{\delta_{2i+1}}$ for $2\le 2i\le g-2$,
\end{itemize} 
where the Dehn twists about the curves on $S'$ are right with respect to the standard orientation.
\end{cor}
\begin{proof}
By the proof of Proposition \ref{HomeoCov}, the restriction of $P$ to $S'$ is a homeomorphism onto $\Sigma$ satisfying the conditions (1,2,3). There is an induced isomorphism $\M(S')\to\M(\Sigma)$, which may be composed with the homomorphism $\M(\Sigma)\to\M(N_{g,n})$ induced by the inclusion $\Sigma\hookrightarrow N_{g,n}$, for any $n\ge 0$, to obtain $\iota$. 
\end{proof}
For any homeomorphism $h\colon N_g\to N_g$ there is a unique orientation preserving lift $\widetilde{h}\colon S_{g-1}\to S_{g-1}$ such that  $h\circ P=P\circ\widetilde{h}$. By \cite{BC}, the mapping $h\mapsto\widetilde{h}$ induces a monomorphism $\theta\colon\M(N_g)\to\M(S_{g-1})$. The following proposition follows from \cite{BC} and \cite[Theorem 10]{SzepB}, where the lift of a crosscap transposition is determined.
\begin{prop}\label{lifts}
There is a monomorphism $\theta\colon\M(N_g)\to\M(S_{g-1})$ such that
\[\theta(t_{\varepsilon_i})=t_{\alpha_{i}}t^{-1}_{\alpha_{g-i}},\quad  \theta(t_{\delta_{2i}})=t_{\beta_{i}}t^{-1}_{\beta_{g-i}},\quad  
\theta(t_{\delta_{2j+1}})=t_{\gamma_{j}}t^{-1}_{\gamma_{g-1-j}},\] for 
$1\le i\le r$, $2\le 2j\le g-2$ and
\[\theta(u_{g-1})=\begin{cases}
t^{-1}_{\beta_r}t_{\beta_{r+1}}(t_{\gamma_r}t_{\beta_r}t_{\beta_{r+1}})^2t^{-1}_{\epsilon}&\textrm{if\ }g=2r+1,\\
t^{-1}_{\gamma_r}t_{\gamma_{r+1}}(t_{\beta_{r+1}}t_{\gamma_r}t_{\gamma_{r+1}})^2t^{-1}_{\phi}&\textrm{if\ }g=2r+2.
\end{cases}\]
\end{prop}
\section{Homological representations}
Fix $g\ge 3$ and let $S=S_{g-1}$, $N=N_g$ and $P\colon S\to N$ be as in the previous section. The group $H_1(S)$ is a free $\Z$-module of rank $2(g-1)$ and the homology classes
$a_i=[\alpha_i]$, $b_i=[\beta_i]$ for $1\le i\le g-1$ form its basis, which is a symplectic basis with respect to the algebraic intersection form:
\[\lr{a_i,a_j}=0,\quad \lr{b_i,b_j}=0,\quad \lr{a_i,b_j}=\delta_{ij}.\]
Let $\Phi\colon\M(S)\to\Sp(H_1(S))$ be the homomorphism induced by the action of $\M(S)$ on $H_1(S)$. If $\gamma$ is an oriented simple closed curve on $S$, 
$[\gamma]\in H_1(S)$ is its homology class, and $t_\gamma$ is the right Dehn twist, then $\Phi(t_\gamma)$ is the transvection 
\begin{equation}\label{transv}
\Phi(t_\gamma)(h)=h+\lr{[\gamma],h}[\gamma],\quad\textrm{for\ }h\in H_1(S).
\end{equation}
From this formula we immediately obtain that, with respect to the basis
$(a_1, b_1,\dots,a_{g-1}, b_{g-1})$, we have
\[\Phi(t_{\alpha_i})=A_i,\quad  \Phi(t_{\beta_i})=B_i,\quad  
\Phi(t_{\gamma_j})=C_j,\]
for $1\le i\le g-1$, $1\le j\le g-2$,  
where $A_i$, $B_i$ and $C_j$ are the matrices defined in Section \ref{nota}.  

The group $H_1(N)$ has the following presentation, as a $\Z$-module:
\[H_1(N)=\lr{x_1,\dots,x_g\,|\,2(x_1+\cdots+x_g)=0},\]
where $x_i=[\xi_i]$. Set $k=x_1+\dots+x_g$ and $R=H_1(N)/\lr{k}$.
Observe that $k$ is the unique element of order two in $H_1(N)$ and $R$ is a free $\Z$-module of rank $g-1$. 

The map $P\colon S\to N$ induces $P_\ast\colon H_1(S)\to H_1(N)$, such that, for $1\le i\le r$
\begin{align*}
&P_\ast(a_i)=x_1+\cdots+x_{2i}=-P_\ast(a_{g-i}),\\ 
&P_\ast(b_i)=x_{2i}+x_{2i+1}=P_\ast(b_{g-i}),
\end{align*}
 and if $g=2r+2$ then
\[P_\ast(a_{r+1})=x_1+\cdots+x_g=k,\quad  
P_\ast(b_{r+1})=2x_g.\]
Let $q\colon H_1(S)\to R$ be the composition of $P_\ast$ with the canonical projection $H_1(N)\to R$, and set $K=\ker q$.
It is easy to verify that $K$ has rank $g-1$ and the following elements form its basis:
\begin{align*}
&e_i=a_i+a_{g-i},\quad e_{r+i}=b_i-b_{g-i}\quad\textrm{for\ }1\le i\le r,\\
&e_{2r+1}=a_{r+1}\quad\textrm{\ for\ }g=2r+2.
\end{align*} 
We also set
\begin{align*}
&f_i=b_i,\quad f_{r+i}=a_{g-i}\quad\textrm{for\ }1\le i\le r,\\
&f_{2r+1}=b_{r+1}\quad\textrm{\ for\ }g=2r+2.
\end{align*} 
Observe that the elements $e_i$, $f_i$ for $1\le i\le g-1$ form a symplectic basis of $H_1(S)$. It follows that $H_1(S)/K$ is a free $\Z$-module of rank $g-1$, which is canonically isomorphic to $R$ if $g$ is odd, or to an index-two subgroup of $R$ if $g$ is even. The group $\M(N)$ acts on $H_1(S)$ by   the composition
$\Phi\circ\theta\colon\M(N)\to\Sp(H_1(S))$. Observe that $K$ is $M(N)$-invariant and hence we have two $(g-1)$-dimensional representations 
\[\psi_1\colon\M(N)\to\GL(K),\quad \psi_2\colon\M(N)\to\GL(H_1(S)/K).\]
\begin{lemma}
$\ker\Psi_1=\ker\Psi_2$ and $\theta(\ker\Psi_1)\subset\ker\Phi$.
\end{lemma}
\begin{proof} Fix the basis $(e_1,\dots,e_{g-1},f_1,\dots,f_{g-1})$ of $H_1(S)$. For any $x\in\M(N)$ let $X$ be the matrix of $\Phi(\theta(x))$. We have
$X=\begin{pmatrix}X_1&Y\\0&X_2\end{pmatrix}$, where $X_1, X_2, Y$ are $(g-1)\times(g-1)$ matrices.
The matrix of the algebraic intersection form is $\Omega=\begin{pmatrix}0&I_{g-1}\\-I_{g-1}&0\end{pmatrix}$ and since $X$ is symplectic, we have $X^t\Omega X=\Omega$, which gives $X_1^tX_2=I$. Therefore $X_1=I\Leftrightarrow X_2=I$, which proves $\ker\Psi_1=\ker\Psi_2$. To prove
the second part of the lemma, assume $X_1=X_2=I$. 
Let $j_\ast\colon\M(S)\to\M(S)$ be the map induced by the covering involution $j$. It is easy to check that the matrix of $j_\ast$ has the form
$J=\begin{pmatrix}-I_{g-1}&T\\0&I_{g-1}\end{pmatrix}$ for some $T$. We have $XJ=JX$, which implies $Y=0$.
\end{proof}
Note that $\ker\Phi$ is the Torelli group, which is well known to be torsion free, and since $\theta$ is a monomorphism, we immediately obtain the following.
\begin{cor}\label{kerTF}
$\ker\Psi_1$ is torsion free.\hfill{$\Box$}
\end{cor}
\begin{rem}
Let $H$ denote the subgroup of $\M(N)$ consisting of the elements inducing the identity on $H_1(N)$. It was proved in \cite{Gas} that $\theta(H)\subset\ker\Phi$.
We leave it as an exercise to check that if $g$ is odd, then $H=\ker\Psi_2$, whereas if $g$ is even, then $H$ is an index-two subgroup of $\ker\Psi_2$. In the latter case, if $g=2r+2$, then we have $\ker\Psi_2=H\cup t_{\varepsilon_{r+1}}H$. 
\end{rem}
\begin{rem}
There is a nontrivial action of $\pi_1(N)$ on $\Z$ defined as follows: $\gamma\in\pi_1(N)$ acts by multiplication by $1$ or $-1$ according to whether $\gamma$ preserves or reverses local orientations of $N$. This action gives rise to homology groups with local coefficients
$H_\ast(N,\widetilde{\Z})$, where $\widetilde{\Z}$ is $\Z$ with the nontrivial 
$\Z[\pi_1(N)]$-module structure.
By \cite[Example 3H.3]{Hat}, we have the  exact sequence 
\[H_2(N)\to H_1(N,\widetilde{\Z})\to H_1(S)\stackrel{P_\ast}{\longrightarrow}H_1(N),\]
which is a part of a long exact sequence of homology groups.
Since $H_2(N)=0$, we have a $\M(N)$-equivariant isomorphism  $H_1(N,\widetilde{\Z})\cong\ker P_\ast$. If $g$ is odd, then $\ker P_\ast=K$, whereas if $g$ is even, then $\ker P_\ast$ is an index-two subgroup of $K$. Therefore the representations $\Psi_1$ and $\Psi_2$ may be seen as coming from the actions of $\M(N)$ on $H_1(N,\widetilde{\Z})$ and $H_1(N)$ respectively.
\end{rem}
For $K$ we fix the basis 
\begin{align*}
&(e_1,e_{r+1},\dots,e_r,e_{2r})\quad\textrm{if\ }g=2r+1,\\
&(e_1,e_{r+1},\dots,e_r,e_{2r},e_{2r+1})\quad\textrm{if\ }g=2r+2.
\end{align*}
For $H_1(S)/K$ we fix the basis
\begin{align*}
&(a_1+K,b_1+K,\dots,a_r+K,b_r+K)\quad\textrm{if\ }g=2r+1,\\
&(a_1+K,b_1+K,\dots,a_r+K,b_r+K,b_{r+1}+K)\quad\textrm{if\ }g=2r+2.
\end{align*}
Having fixed bases for $K$ and $H_1(S)/K$ we can now compute, for $\Psi_1$ and $\Psi_2$, the images of the generators of $\M(N)$. This is done by a straightforward calculation, using Proposition \ref{lifts} and the formula (\ref{transv}).
For $k=1,2$ and  $1\le i\le r$, $1\le j\le r-1$ we have
\[\Psi_k(t_{\varepsilon_i})=A_i,\quad  \Psi_k(t_{\delta_{2i}})=B_i,\quad  
\Psi_k(t_{\delta_{2j+1}})=C_j.\] 
 If $g=2r+1$ then
\[\Psi_1(u_{g-1})=\begin{pmatrix}I_{g-3}&0&0\\0&1&0\\0&1&-1\end{pmatrix},\quad
\Psi_2(u_{g-1})=\begin{pmatrix}I_{g-3}&0&0\\0&-1&0\\0&-1&1\end{pmatrix}.
\]
If $g=2r+2$ then 
\[\Psi_1(t_{\delta_{g-1}})=\begin{pmatrix}I_{g-4}&0&0&0\\0&1&1&0\\0&0&1&0\\0&0&-2&1\end{pmatrix},\quad
\Psi_2(t_{\delta_{g-1}})=\begin{pmatrix}I_{g-4}&0&0&0\\0&1&1&-2\\0&0&1&0\\0&0&0&1\end{pmatrix},
\]
\[\Psi_1(u_{g-1})=\begin{pmatrix}I_{g-4}&0&0&0\\0&1&-1&1\\0&0&1&0\\0&0&2&-1\end{pmatrix},\quad
\Psi_2(u_{g-1})=\begin{pmatrix}I_{g-4}&0&0&0\\0&1&1&-2\\0&0&1&0\\0&0&1&-1\end{pmatrix}.
\]
Now it is easy to see that $\Psi_1$ and $\Psi_2$ are not conjugate as homomorphism to $\GL(g-1,\C)$. For suppose that there is $M\in\GL(g-1,\C)$, such that $\Psi_1(x)=M\Psi_2(x)M^{-1}$ for all $x\in\M(N)$. Then $M$ commutes with
$A_i$, $B_i$, $C_j$ for $1\le i\le r$, $1\le j\le r-1$, and by Lemma \ref{diag},
$M=\alpha I_{2r}$ if $g=2r+1$, or $M=\diag(\alpha I_{2r},\beta)$ if $g=2r+2$,
for $\alpha,\beta\in\C$.
In either case it is impossible that $\Psi_1(u_{g-1})=M\Psi_2(u_{g-1})M^{-1}$.
\section{Homomorphisms from $\M(N_{g,n})$ to $\GL(m,\C)$ for $m<g-1$}
The aim of this section is to prove Theorem \ref{MNtoGLfact}. The proof is divided in two parts. 

\begin{proof}[Proof of Theorem \ref{MNtoGLfact} for $(g,m)\ne(6,4)$.]
Suppose that $n\in\{0,1\}$,  $g=2r+s$ for $r\ge 2$, $s\in\{1,2\}$, $m\le g-2$ and $f\colon\M(N_{g,n})\to\GL(m,\C)$ is a homomorphism. By Theorem \ref{abNg}, it suffices to prove that $\im(f)$ is abelian.
Let $S'=S_{r,s}$ and 
$\iota\colon\M(S')\to\M(N_{g,n})$ be the homomorphism from Corollary \ref{HomSN}.
Set $f'=f\circ\iota$ and observe that if $\im(f')$ is abelian, then so is $\im(f)$, by Lemma \ref{com_norm}.

Suppose that $m\le 2r-1$. Then $\im(f')$ is either trivial or cyclic by Theorem \ref{FHK}  and we are done. This finishes the proof for odd $g$.

Suppose that $g=2r+2$ for $r\ge 3$ and $m=2r$. By Theorem \ref{KorU}, $f'$ is either trivial or conjugate to the homological representation $\Phi$. In the former case we are done.  
In the latter case, by the definition of $\Phi$ we have $\Phi(t_{\gamma_r})=\Phi(t_{\alpha_r})$ because the curves $\gamma_r$ and $\alpha_r$ become isotopic after gluing discs to the boundary of $S'$. It follows that
$f(t_{\delta_{2r+1}})=f(t_{\varepsilon_r})$ and by Lemma \ref{tsq} $f(t_{\delta_1}^2)=1$. This is a contradiction because $\Phi(t_{\alpha_1})$ has infinite order.
\end{proof}
In order to prove Theorem \ref{MNtoGLfact} for $(g,m)=(6,4)$, we first prove some lemmas.
\begin{lemma}\label{MN4toGL2}
Suppose that $f\colon\M(N_{4,n})\to\GL(2,\C)$ is a homomorphism.
Then, with respect to some basis one of the following cases holds.
\begin{itemize}
\item[(1)] $f(t_{\delta_1})=f(t_{\delta_2})=f(t_{\delta_3})=\lambda I$, $\lambda\in\{-1,1\}$
\item[(2)] $f(t_{\delta_1})=f(t_{\delta_2})=f(t_{\delta_3})=\begin{pmatrix}1&0\\0&-1\end{pmatrix}$ 
\item[(3)] $f(t_{\delta_1})=f(t_{\delta_3})=\begin{pmatrix}1&1\\0&-1\end{pmatrix}$, $f(t_{\delta_2})=\begin{pmatrix}-1&0\\1&1\end{pmatrix}$. 
\end{itemize}
In particular $f(t_{\delta_1}^2)=1$.
\end{lemma}
\begin{proof}
For $i=1,2,3$ let $L_i=f(t_{\delta_i})$ and $U=f(u_{3})$. 

Suppose that $L_1$ has only one eigenvalue $\lambda$. Since $L_1$ is conjugate to $L_1^{-1}$ (by (R12)), we have $\lambda\in\{-1,1\}$. If $\dim E(L_1,\lambda)=2$, then we have the case (1). Suppose that $\dim E(L_1,\lambda)=1$. If $E(L_1,\lambda)\ne E(L_2,\lambda)$, then with respect to some basis we have
$L_1=\begin{pmatrix}\lambda&1\\0&\lambda\end{pmatrix}$, $L_2=\begin{pmatrix}\lambda&0\\x&\lambda\end{pmatrix}$ for some $x$, and from 
the braid relation $L_1L_2L_1=L_2L_1L_2$ we have $x=-1$. Since $L_3$ commutes with $L_1$ we have
$L_3=\begin{pmatrix}\lambda&y\\0&\lambda\end{pmatrix}$ for some $y$, and from
$L_2L_3L_2=L_3L_2L_3$ we obtain $y=1$, hence $L_1=L_3$. Since $\delta_1=\varepsilon_1$, we have $L_1^2=I$ by Lemma \ref{tsq} (for $i=1$, $j=3$), which is a contradiction. If $E(L_1,\lambda)=E(L_2,\lambda)$, then with respect to some basis we have
$L_1=\begin{pmatrix}\lambda&1\\0&\lambda\end{pmatrix}$, 
$L_2=\begin{pmatrix}\lambda&x\\0&\lambda\end{pmatrix}$, and it is easy to obtain a contradiction as above, by showing that $L_1=L_2=L_3$.

Suppose that $L_1$ has two eigenvalues $\lambda, \mu$. Then with respect to some basis we have $L_1=\begin{pmatrix}\lambda&0\\0&\mu\end{pmatrix}$, and since $L_3$ and $U$ commute with $L_1$, they are also diagonal. In particular we have $UL_3=L_3U$ and $L_3UL_3=U$ (R12) gives $L_3^2=1$, which implies $\{\lambda, \mu\}=\{-1,1\}$. We have $L_3=L_1$ or $L_3=-L_1$. In the latter case the braid relations $L_3L_2L_3=L_2L_3L_2$ and $L_1L_2L_1=L_2L_1L_2$ imply $L_2L_1L_2=0$, a contradiction, hence $L_1=L_3$.

If $E(L_1,1)\ne E(L_2,1)$, then with respect to some basis we have 
$L_1=\begin{pmatrix}1&1\\0&-1\end{pmatrix}$, 
$L_2=\begin{pmatrix}-1&0\\x&1\end{pmatrix}$. From $L_1L_2L_1=L_2L_1L_2$ we have $x=1$ and we are in the case (3). Analogously, if $E(L_1,-1)\ne E(L_2,-1)$, then with respect to some basis we have 
$L_1=\begin{pmatrix}-1&1\\0&1\end{pmatrix}$, 
$L_2=\begin{pmatrix}1&0\\1&-1\end{pmatrix}$, and since $E(L_1,1)\ne E(L_1,1)$, we are in the case (3) again.

Finally, if $E(L_1,1)=E(L_2,1)$ and $E(L_1,-1)=E(L_2,-1)$, then with respect to some basis we have 
$L_1=L_2=\begin{pmatrix}1&0\\0&-1\end{pmatrix}$ and we are in the case (2).
\end{proof}
\begin{lemma}\label{MN6toGL4a^2}
Suppose that $n\le 1$ and $f\colon\M(N_{6,n})\to\GL(4,\C)$ is a homomorphism such that $f(t_{\delta_1}^2)=1$. Then $\im(f)$ is abelian.
\end{lemma}
\begin{proof}
Let $H$ be the normal closure of $t_{\delta_1}^2$ in $\M(N_{6,n})$ and set 
$G=\M(N_{6,n})/H$. We have an induced homomorphism $f'\colon G\to\GL(4,\C)$ such that
$f=f'\circ\pi$, where $\pi\colon\M(N_{6,n})\to G$ is the canonical projection. By the relations (R1, R4), the mapping $\rho(\sigma_i)=\pi(t_{\delta_i})$,
where $\sigma_i$ is the transposition $(i,i+1)$ for $1\le i\le 5$, defines a 
homomorphism $\rho\colon\mathfrak{S}_6\to G$. Let 
$\phi\colon\mathfrak{S}_6\to\GL(4,\C)$ be the composition $f'\circ \rho$. 
 By Lemma \ref{repsym}, $\phi$ is the direct sum of one-dimensional representations. In particular the image of $\phi$ is abelian,
and so is $\im(f)$ by Lemma \ref{com_norm}.
\end{proof}

Let $R$ be the subsurface obtained by removing from $N_{6,n}$ a regular neighbourhood of $\delta_1\cup\delta_2$. Note that $R$ is homeomorphic to $N_{4,n+1}$. The homomorphism $\M(R)\to\M(N_{6,n})$ induced by the inclusion of $R$ in $N_{6,n}$ is injective, and we will treat $\M(R)$ as a subgroup of $\M(N_{6,n})$.

\begin{lemma}\label{MN6toGL4splitting}
Suppose that $n\le 1$, $f\colon\M(N_{6,n})\to\GL(4,\C)$ is a homomorphism and there exists a splitting $\mathbb{C}^4=V_1\oplus V_2$ such that $V_i$ is a $2$-dimensional $\M(R)$-invariant subspace for $i=1,2$. Then $\im(f)$ is  abelian.\end{lemma}
\begin{proof}
Let $f'$ be the restriction of $f$ to $\M(R)$. With respect to the splitting $\mathbb{C}^4=V_1\oplus V_2$ we have $f'=f_1\oplus f_2$ for some
$f_i\colon\M(R)\to\GL(2,\C)$, $i=1,2$. By  Lemma \ref{MN4toGL2} we have $f_i(t_{\delta_4}^2)=1$ for $i=1,2$, hence $f(t_{\delta_4}^2)=1$ and we are done by Lemma \ref{MN6toGL4a^2}.
\end{proof}

\begin{lemma}\label{MN6toGL4inv}
Suppose that $n\le 1$, $f\colon\M(N_{6,n})\to\GL(4,\C)$ is a homomorphism, $f(t_{\delta_1})$ has only one eigenvalue and there exists a $2$-dimensional $\M(R)$-invariant subspace. Then $\im f$ is abelian.
\end{lemma}
\begin{proof}
Let $\lambda$ be the eigenvalue of $f(t_{\delta_1})$. 
Fix a basis of $\mathbb{C}^4$ whose first two vectors span the $\M(R)$-invariant subspace. By the case (1) of Lemma \ref{MN4toGL2}, with respect to such basis we have
$f(t_{\delta_4})=\begin{pmatrix}\lambda I&X\\0&\lambda I\end{pmatrix}$, 
$f(t_{\delta_5})=\begin{pmatrix}\lambda I&Y\\0&\lambda I\end{pmatrix}$, for some $2\times 2$ matrices $X, Y$. In particular $f(t_{\delta_4})$ and $f(t_{\delta_5})$ commute and we are done by Lemma \ref{com_norm}.
\end{proof}

\begin{proof}[Proof of Theorem \ref{MNtoGLfact} for $g=6$, $m=4$.]
Suppose that $n\in\{0,1\}$ and $f\colon\M(N_{6,n})\to\GL(4,\C)$ is a homomorphism. For $1\le i\le 5$ we set $L_i=f(t_{\delta_i})$ and
$M=f(t_{\varepsilon_2})$, $U_5=f(u_5)$.
We consider the following cases.
\begin{itemize}
\item[(1)] $L_1$ has $4$ eigenvalues. 
\item[(2)] $L_1$ has $3$ eigenvalues. 
\item[(3)] $L_1$ has $2$ eigenvalues with equal multiplicities. 
\item[(4)] $L_1$ has $2$ eigenvalues with different multiplicities. 
\item[(5)] $L_1$ has $1$ eigenvalue. 
\end{itemize}
In the cases (1, 2, 3) it is easy to find a splitting $\mathbb{C}^4=V_1\oplus V_2$ such that $V_i$ is a $2$-dimensional $\M(R)$-invariant subspace for $i=1,2$. For example, suppose that $L_1$ has three eigenvalues $\lambda_1, \lambda_2, \lambda_3$ such that $\#\lambda_1=\#\lambda_2=1$ and $\#\lambda_3=2$. Then we take
$V_1=E(L_1,\lambda_1)\oplus E(L_1,\lambda_2)$ and
$V_2=E(L_1,\lambda_3)$ if $\dim E(L_1,\lambda_3)=2$ or
$V_2=E^2(L_1,\lambda_3)$ if $\dim E(L_1,\lambda_3)=1$. Therefore in the cases (1, 2, 3) we are done by Lemma \ref{MN6toGL4splitting}.

Assume (5). Let $\lambda$ be the unique eigenvalue of
 $L_1$  and  $k=\dim E(L_1,\lambda)$.
If $k=4$ then $L_1=\lambda I$ and the image of $f$ is cyclic. If $k=2$ or $k=1$ then respectively $E(L_1,\lambda)$ or $E^2(L_1,\lambda)$ is a $2$-dimensional $\M(R)$-invariant subspace, and we are done by Lemma \ref{MN6toGL4inv}. Suppose that $k=3$. 
If $E(L_1,\lambda)\ne E(L_2,\lambda)$ then $E(L_1,\lambda)\cap E(L_2,\lambda)$ 
is a $2$-dimensional $\M(R)$-invariant subspace, and we are done by Lemma \ref{MN6toGL4inv}. If $E(L_1,\lambda)=E(L_2,\lambda)$ then 
with respect to some basis we have
\[L_1=\begin{pmatrix}\lambda&0&0&0\\0&\lambda&0&0\\0&0&\lambda&1\\0&0&0&\lambda\end{pmatrix},\quad
L_2=\begin{pmatrix}\lambda&0&0&x\\0&\lambda&0&y\\0&0&\lambda&z\\0&0&0&\lambda\end{pmatrix}.\]
In particular $L_1$ and $L_2$ commute and we are done by Lemma \ref{com_norm}.

It remains to consider the case (4). Suppose that $L_1$ has eigenvalues $\mu$, $\lambda$, with $\#\mu=1$ and $\#\lambda=3$. Since $L_1$ is conjugate to $L_1^{-1}$, we have $\{\mu,\lambda\}=\{-1,1\}$. It follows from Theorem \ref{abNg} that there is a homomorphism $\tau(\M(N_6))\to\{-1,1\}$ such that $\tau(a_1)=-1$. By multiplying $f$ by $\tau$ if necessary, we may assume  $\mu=-1$, $\lambda=1$. The Jordan form of $L_1$ is one of the following three matrices.
\[(i)\ \begin{pmatrix}-1&0&0&0\\0&1&0&0\\0&0&1&0\\0&0&0&1\end{pmatrix}\quad 
(ii)\ \begin{pmatrix}-1&0&0&0\\0&1&1&0\\0&0&1&1\\0&0&0&1\end{pmatrix}\quad 
(iii)\ \begin{pmatrix}-1&0&0&0\\0&1&0&0\\0&0&1&1\\0&0&0&1\end{pmatrix}
\]

In the case (i) we have $L_1^2=I$ and we are done by Lemma \ref{MN6toGL4a^2}. 

In the case (ii) the following subspaces are $\M(R)$-invariant:
$E(L_1,-1)$, $E(L_1,1)$, $E^2(L_1,1)$, $E^3(L_1,1)$. It follows that 
\[M=\begin{pmatrix}x_1&0&0&0\\0&x_2&v_1&v_2\\0&0&x_3&v_3\\0&0&0&x_4\end{pmatrix},\quad
L_4=\begin{pmatrix}y_1&0&0&0\\0&y_2&w_1&w_2\\0&0&y_3&w_3\\0&0&0&y_4\end{pmatrix}.\]
The braid relation $ML_4M=L_4ML_4$ (R5) implies $x_i=y_i$ for $1\le i\le 4$. Since the first two vectors of the basis are eigenvectors of $M$, they have to correspond to different eigenvalues of $M$. Therefore $x_2=-x_1$, $x_3=x_4=1$ and $x_1=1$ or $x_1=-1$. In either case it is not difficult to check that   $ML_4M=L_4ML_4$ holds if and only if $M=L_4$. We are done by Lemma \ref{com_norm}.

In the case (iii) the following subspaces are $\M(R)$-invariant:
$E(L_1,-1)$, $E(L_1,1)$, $E^2(L_1,1)$. We have $\dim E(L_1,1)=2$ and by  applying Lemma \ref{MN4toGL2} to the action of $\M(R)$ on this subspace, we obtain three sub-cases.

Sub-case (iiia). If the action of $\M(R)$ on $E(L_1,1)$ is trivial, then we have
\[M=\begin{pmatrix}-1&0&0&0\\0&1&0&x_1\\0&0&1&x_2\\0&0&0&1\end{pmatrix},\quad
L_4=\begin{pmatrix}-1&0&0&0\\0&1&0&y_1\\0&0&1&y_2\\0&0&0&1\end{pmatrix}\]
As in the case (ii), the braid relation implies $M=L_4$ and we are done by Lemma \ref{com_norm}.

Sub-case (iiib). By changing the basis of $E(L_1,1)$ we may assume that
\[M=\begin{pmatrix}1&0&0&0\\0&-1&0&x_1\\0&0&1&x_2\\0&0&0&1\end{pmatrix},\quad
L_4=\begin{pmatrix}1&0&0&0\\0&-1&0&y_1\\0&0&1&y_2\\0&0&0&1\end{pmatrix}\]
As in the case (ii), the braid relation implies $M=L_4$ and we are done by Lemma \ref{com_norm}.

Sub-case (iiic). By changing the basis of $E(L_1,1)$ we may assume that
\[M=\begin{pmatrix}1&0&0&0\\0&1&1&x_1\\0&0&-1&x_2\\0&0&0&1\end{pmatrix},
L_4=\begin{pmatrix}1&0&0&0\\0&-1&0&y_1\\0&1&1&y_2\\0&0&0&1\end{pmatrix},
L_5=\begin{pmatrix}1&0&0&0\\0&1&1&z_1\\0&0&-1&z_2\\0&0&0&1\end{pmatrix}.\]
By solving the equations $ML_4M=L_4ML_4$ and $L_5L_4L_5=L_4L_5L_4$ we obtain $x_2=-(2x_1+y_1+2y_2)$, $z_2=-(2z_1+y_1+2y_2)$, and from $ML_5=L_5M$ we obtain $x_2=z_2$. Thus $M=L_5$ and by Lemma \ref{tsq} 
$L_1^2=1$. We are done by Lemma \ref{MN6toGL4a^2}. 
\end{proof}
\section{Homomorphisms between mapping class groups}
The aim of this section is to prove Theorem \ref{MNgtoMNh}.
Fix $g\ge 5$ and set $\M=\M(N_g)$. We are going to use the fact that $s=t_{\delta_1}\cdots t_{\delta_{g-1}}$ has finite order in $\M$ (equal to $g$ if it is even, or $2g$ otherwise, see \cite{PSz}). 
By the relations (R1,R4) we have 
\begin{equation}\label{srel}
t_{\delta_{i+1}}s=st_{\delta_i}\quad\textrm{for\ }1\le i\le g-2.
\end{equation}
By Theorem \ref{abNg} we have $s\in[\M,\M]$ for $g\ge 7$ and $g=5$, $s^2\in[\M,\M]$ for $g=6$.
\begin{proof}[Proof of Theorem \ref{MNgtoMNh}]
Suppose that $g\ge 5$, $h<g$ and $f\colon\M(N_g)\to\M(N_h)$ is a homomorphism.
Since $M(N_h)$ is abelian for $h\le 2$, we are assuming $h\ge 3$.

Let $f'\colon\M(N_{g})\to\GL(h-1,\C)$ be the composition $\Psi_1\circ f$ and
$K=\ker \Psi_1$.
By Theorem \ref{MNtoGLfact}, $\im(f')$ is abelian, hence $f([\M(N_g),\M(N_g)])\subseteq K$. 
Suppose that $g\ge 7$ or $g=5$. Then $f(s)\in K$, and since $K$ is torsion free by Lemma \ref{kerTF}, thus $f(s)=1$. This gives, by (\ref{srel}), $f(t_{\delta_1})=f(t_{\delta_2})$ and we are done by Lemma \ref{com_norm}. If $g=6$ then $f(s^2)\in K$, which gives 
$f(s^2)=1$ and $f(t_{\delta_2})=f(t_{\delta_4})$. Since $t_{\delta_1}$ commutes with $t_{\delta_4}$, thus $f(t_{\delta_1})$ commutes with $f(t_{\delta_2})$  and we are done by Lemma \ref{com_norm}.
\end{proof}

Note that Theorems \ref{MNtoGLfact} and \ref{MNgtoMNh} are trivially true for $g\le 3$ because
$\GL(1,\C)=\C^\ast$, $\M(N_2)\cong\Z_2\times\Z_2$, $M(N_1)=1$ are abelian groups.
On the other hand, Corollary \ref{g4cex} below shows that both theorems are false for $g=4$ (recall that $\M(N_3)\cong\GL(2,\Z)$).
Let $D_\infty$ denote the infinite dihedral group, defined by the presentation
\[D_\infty=\lr{x,y\,|\,x^2=y^2=1}.\]
\begin{lemma}\label{MN4onD}
There is an epimorphism $\phi\colon\M(N_4)\to D_\infty$. 
\end{lemma}
\begin{proof}
According to the main result of \cite{Szepg4} simplified in \cite{PSz}, $\M(N_4)$ admits a presentation with generators $t_{\varepsilon_2}$,  $t_{\delta_i}$, $u_i$ for $i=1,2,3$ and relations (R1, R3, R4, R6, R7, R9, R10, R11, R12) and
\begin{align*}
&t_{\delta_{i+1}}u_1u_{i+1}=u_iu_{i+1}t_{\delta_i}\quad\mathrm{for\ }i=1,2\\
&(t_{\varepsilon_2}u_3)^2=1,\quad t_{\delta_1}(t_{\delta_2}t_{\delta_3}u_3u_2)t_{\delta_1}=(t_{\delta_2}t_{\delta_3}u_3u_2).
\end{align*}
It is easy to check that the mapping $\phi(t_{\varepsilon_2})=xy$,  $\phi(t_{\delta_i})=1$, $\phi(u_i)=y$ for $i=1,2,3$, respects the defining relations of $\M(N_4)$, hence it defines a homomorphism onto $D_\infty$.
\end{proof}
\begin{cor}\label{g4cex}
For $h\ge 3$ there is a homomorphism $f\colon\M(N_4)\to\M(N_h)$, such that  $\im(f)$ is isomorphic to $D_\infty$.
\end{cor}
\begin{proof}
Fix $h\ge 3$. By the proof of \cite[Theorem 3]{SzepCR}, $t_{\delta_1}$ can be written in $\M(N_h)$ as a product of two involutions $\sigma, \tau$. Since $t_{\delta_1}$ has infinite order in $\M(N_h)$, the mapping $x\mapsto\sigma$, $y\mapsto\tau$ defines an embedding $D_\infty\to\M(N_h)$. By pre-composing this embedding with the epimorphism $\phi$ from Lemma \ref{MN4onD}, we obtain $f$. 
\end{proof}
The following two theorems can be proved by the same method as Theorem \ref{MNgtoMNh}. We leave the details to the reader.
\begin{theorem}
Suppose that $g\ge 5$, $g\ge 2h+2$ and $f\colon\M(N_g)\to\M(S_h)$ is a homomorphism. Then $\im(f)$ is abelian.
\end{theorem}
\begin{theorem}
Suppose that $g\ge 3$ and $h\le 2g$. Then the only homomorphism from $\M(S_g)$ to $\M(N_h)$ is the trivial one.
\end{theorem}

\section{Homomorphisms from $\M(N_g)$ to $\GL(g-1,\C)$.}
The aim of this section is to prove Theorem \ref{MNtoGLg-1}. The prove it divided in two cases, according to the parity of the genus.

Let $g=2r+s$, $s\in\{1,2\}$,  $S'=S_{r,s}$ and 
$\iota\colon\M(S')\to\M(N_{g,n})$ be the homomorphism from Corollary \ref{HomSN}.
If $f\colon\M(N_{g,n})\to\GL(m,\C)$ is a homomorphism, then we set $f'=f\circ\iota$. 

\begin{proof}[Proof of Theorem \ref{MNtoGLg-1} for odd $g$.]
Suppose that $N=N_{2r+1}$, $r\ge 3$ and $f\colon\M(N)\to\GL(2r,\C)$ is a homomorphism, such that $\im(f)$ is not abelian.
By Theorem \ref{KorU}, $f'$ is conjugate to the homological representation $\Phi$, and thus  the exists a basis, such that
$f(t_{\varepsilon_i})=f'(t_{\alpha_i})=A_i$, $f(t_{\delta_{2i}})=f'(t_{\beta_i})=B_i$ for $1\le i\le r$ and
$f(t_{\delta_{2j+1}})=f'(t_{\gamma_j})=C_j$ for $1\le j\le r-1$.
Set $U_k=f(u_k)$ for $1\le k\le 2r$.

Since $U_{2r}$ commutes with $A_i$ and $B_i$ for $1\le i\le r$, and with $C_j$ for $j=1,\dots,r-2$ (R6,R8) thus, by Lemma \ref{diag},  \[U_{2r}=\begin{pmatrix}\lambda I_{2r-2}&0\\0&X\end{pmatrix},\]
for some $2\times 2$ matrix $X$. Since $U_{2r}$ is conjugate to $U_{2r}^{-1}$ we have $\lambda\in\{-1,1\}$ and by multiplying $f$ by $(-1)^\ab$ if necessary, we may assume $\lambda=1$. The relation $B_rU_{2r}B_r=U_{2r}$ (R12) implies 
$X=\begin{pmatrix}x&0\\y&-x\end{pmatrix}$.
From (R11) and (R7) we have 
\begin{align*}
&U_{2r-2}=(C_{r-1}B_rB_{r-1}C_{r-1})^{-1}U_{2r}(C_{r-1}B_rB_{r-1}C_{r-1}),\\
&U_{2r}U_{2r-2}-U_{2r-2}U_{2r}=0,
\end{align*}
and since the left hand side of the last equation is equal to
\[(1-x^2)(E_{2r,2r-3}+E_{2r-2,2r-1}),\] thus $x^2=1$. We have $U_{2r}^{-1}=U_{2r}$, and from (R11) and (R9)
\begin{align*}
&U_{2r-1}=(C_{r-1}B_r)^{-1}U_{2r}(C_{r-1}B_r),\\  
&U_{2r}U_{2r-1}U_{2r}-U_{2r-1}U_{2r}U_{2r-1}=0.
\end{align*}
By considering the cases $x=1$ and $x=-1$ separately, we find that the left hand side of the last equation is of the form $(y-x)^2Z$, where  $Z\ne 0$. Hence $x=y$ and $U_{2r}=\Psi_1(u_{2r})$ if $x=1$, or $U_{2r}=\Psi_2(u_{2r})$ if $x=-1$.
By Theorem \ref{gener}, $f$ is equal to $\Psi_1$ or $\Psi_2$ on generators of $\M(N)$.
\end{proof}

Now we will borrow some arguments from \cite{KorkSymp} to prove Lemma \ref{AB} below, which will be a starting point for the proof of Theorem \ref{MNtoGLg-1} for even genus.

\begin{lemma}\label{flag}
Suppose that $n\le 1$, $g\ge 5$ and $f\colon\M(N_{g,n})\to\GL(m,\C)$ is a homomorphism. If there is a flag $0=W_0\subset W_1\subset\cdots\subset W_k=\mathbb{C}^m$ of $\M(N_{g,n})$-invariant subspaces such that
$\dim(W_i/W_{i-1})<g-1$ for $i=1,\dots,k$, then $\im(f)$ is abelian.
\end{lemma}
\begin{proof}
The same argument as in the proof of \cite[Lemma 4.8]{KorkSymp} can be applied, using Theorem \ref{MNtoGLfact}, to
show that with respect to some basis $f[\M(N_{g,n}),\M(N_{g,n})]$ is contained in the subgroup of upper triangular matrices with $1$ on the diagonal. Since this subgroup is nilpotent and $[\M(S'),\M(S')]$ is perfect, it follows that $f'[\M(S'),\M(S')]$ is trivial, which means that $\im(f')$ is abelian, and so is  $\im(f)$.
\end{proof}

\begin{lemma}\label{dim_eigen}
Suppose that $N=N_{2r+2}$, $r\ge 3$ and $f\colon\M(N)\to\GL(2r+1,\C)$ is a homomorphism, such that $\im(f)$ is not abelian. Then $L_1=f(t_{\delta_1})$ has an eigenvalue $\lambda$ such that $\dim E(L_1,\lambda)=2r$. 
\end{lemma}
\begin{proof}
By \cite[Corollary 4.6]{KorkSymp} applied to $f'$, $L_1$ has at most two eigenvalues. It follows that there is an eigenvalue $\lambda$ with $\#\lambda\ge r+1\ge 4$. Set $m=\dim E(L_1,\lambda)$. Since $\im(f)$ is not abelian, thus $m\le 2r$. We are going to show that $m=2r$.

Let $R$ be the subsurface obtained by removing from $N$ a regular neighbourhood of $\delta_1\cup\delta_2$. We have $R\approx N_{2r,1}$. We treat $\M(R)$ as a subgroup of $\M(N)$.

Suppose $m\le 2r-2$.
Let
$W=E^k(L_1,\lambda)$, where $k=\max\{4-m, 1\}$.
Observe that $W$ is a $\M(R)$-invariant subspace with
$3\le\dim W\le 2r-2$. By Lemma \ref{flag},  $f(\M(R))$ is abelian, which means $f(t_{\delta_4})=f(t_{\delta_5})$. By Lemma \ref{com_norm}, $\im(f)$ is abelian,  a contradiction.

Suppose that $m=2r-1$ and set $L_2=f(t_{\delta_2})$. If $E(L_1,\lambda)\ne E(L_2,\lambda)$ then $E(L_1,\lambda)\cap E(L_2,\lambda)$ is a $\M(R)$-invariant subspace of dimension $2r-3$ or $2r-2$ and we can use the same argument as above to obtain a contradiction. If $E(L_1,\lambda)=E(L_2,\lambda)$, then by \cite[Lemma 4.3]{KorkSymp} applied to $f'$, $E(L_1,\lambda)$ is a $\M(S')$-invariant subspace of dimension $2r-1$, and by \cite[Lemma 4.8]{KorkSymp} $f'$ is trivial. It follows that $\im f$ is abelian, a contradiction.
\end{proof}

\begin{lemma}\label{AB}
Suppose that $N=N_{2r+2}$, $r\ge 3$ and $f\colon\M(N)\to\GL(2r+1,\C)$ is a homomorphism. If $r=3$ then assume that $1$ is the unique eigenvalue of $f(t_{\delta_1})$. Then either $\im(f)$ is abelian, or with respect to some basis
$f(t_{\varepsilon_i})=A_i$, $f(t_{\delta_{2i}})=B_i$ for $i=1,\dots,r$.
\end{lemma}
\begin{proof}
Suppose that $\im(f)$ is not abelian. 
By Lemma \ref{dim_eigen}, $L_1=f(t_{\delta_1})$ has an eigenvalue $\lambda$ with
$\dim E(L_1,\lambda)=2r$. If $r=3$ then $\lambda=1$ by assumption, and for $r\ge 4$, $\lambda=1$ by the proof of \cite[Lemma 5.2]{KorkSymp}. Since $\M(S')$ is perfect, thus $\det L_1=1$ and $\lambda=1$ is the unique eigenvalue. Set $L_2=f(t_{\delta_2})$. 
We claim that $E(L_1,1)\ne E(L_2,1)$. For otherwise it is easy to prove that $L_1$ and $L_2$ commute (see the proof of Theorem \ref{MNtoGLfact} for $(g,m)=(6,4)$, case (5)),  and $\im(f)$ is abelian by Lemma \ref{com_norm},
 a contradiction. Now we can apply \cite[Lemma 4.7]{KorkSymp} to $f'$ to conclude that with respect to some basis we have
$f(t_{\varepsilon_i})=f'(t_{\alpha_i})=A_i$, $f(t_{\delta_{2i}})=f'(t_{\beta_i})=B_i$ for $i=1,\dots,r$.
\end{proof}

\begin{proof}[Proof of Theorem \ref{MNtoGLg-1} for even $g$.]
Suppose that $N=N_{2r+2}$, $r\ge 4$ and $f\colon\M(N)\to\GL(2r+1,\C)$ is a homomorphism, such that $\im(f)$ is not abelian. 
By Lemma \ref{AB} there is a basis such that 
$f(t_{\varepsilon_i})=A_i$ and $f(t_{\delta_{2i}})=B_i$ for $1\le i\le r$.
Set $D_i=f(t_{\delta_{2i+1}})$ for $1\le i\le r$ and
$U_j=f(u_j)$ for $1\le j\le 2r+1$.

Fix $i\in\{1,\dots,r-1\}$. Since $D_i$ is conjugate to $A_1$, it has one eigenvalue $\lambda=1$. For $j\notin\{i,i+1\}$ the relations $D_iA_j=A_jD_i$ and $D_iB_j=B_jD_i$ imply, by Lemma \ref{diag}, that $D_i$ has the form 
\[D_i=\begin{pmatrix}
I_{2(i-1)}&0&0&0&0\\
0&F_{11}&F_{12}&0&X_1\\
0&F_{21}&F_{22}&0&X_2\\
0&0&0&I_{2(g-i-1)}&0\\
0&Y_1&Y_2&0&z
\end{pmatrix},\]
where $F_{kl}$ are $2\times 2$ matrices, $X_k$ are $2\times 1$ vectors,
$Y_l$ are $1\times 2$ vectors and $z$ is a complex number. The relations
$D_iA_i=A_iD_i$ and $D_iA_{i+1}=A_{i+1}D_i$ imply, for $k,l\in\{1,2\}$,  $VF_{kl}=F_{kl}V$,  $VF_{kl}=F_{kl}$ for $k\ne l$, $VX_k=X_k$, $Y_lV=Y_l$, hence
\begin{align*}
&F_{11}=\begin{pmatrix}s_1&t_1\\0&s_1\end{pmatrix}, 
F_{12}=\begin{pmatrix}0&v_1\\0&0\end{pmatrix}, 
X_1=\begin{pmatrix}x_1\\0\end{pmatrix},\\
&F_{21}=\begin{pmatrix}0&v_2\\0&0\end{pmatrix}, 
F_{22}=\begin{pmatrix}s_2&t_2\\0&s_2\end{pmatrix},
X_2=\begin{pmatrix}x_2\\0\end{pmatrix},\\ 
&Y_1=\begin{pmatrix}0&y_1\end{pmatrix}, 
Y_2=\begin{pmatrix}0&y_2\end{pmatrix}. 
\end{align*}
Since $s_1$, $s_2$ are eigenvalues, we have $s_1=s_2=1$ and $\det D_i=z$, which gives $z=1$. Now, by solving the equations $B_iD_iB_i-D_iB_iD_i=0$ and $B_{i+1}D_iB_{i+1}-D_iB_{i+1}D_i=0$ we obtain
$t_1=t_2=1$, $v_1v_2=1$, $y_2=y_1v_1$, $x_2=x_1v_2$, $x_1y_1=0$.
Thus, for $i=1,\dots,r-1$ we have
\[D_i=\begin{pmatrix}
I_{2(i-1)}&0&0&0&0&0&0\\
0&1&1&0&\alpha_i&0&\alpha_ix_i\\
0&0&1&0&0&0&0\\
0&0&\alpha_i^{-1}&1&1&0&x_i\\
0&0&0&0&1&0&0\\
0&0&0&0&0&I_{2(g-i-1)}&0\\
0&0&y_i&0&\alpha_iy_i&0&1
\end{pmatrix},\quad x_iy_i=0.\]
Similarly, using the relations between $D_r$ and $A_i$, $B_i$ it can be shown that
\[D_r=\begin{pmatrix}I_{2g-2}&0&0&0\\
0&1&1&x_r\\
0&0&1&0\\
0&0&y_r&1\end{pmatrix},\quad x_ry_r=0.\]
It is not possible that $x_r=y_r=0$, because then $D_r=A_r$ and Lemma \ref{tsq} would give a contradiction.
For $1\le i\le r-1$, by solving the equation $D_iD_r-D_rD_i=0$ we obtain $x_iy_r=0$ and $x_ry_i=0$. It follows that either $x_i=0$ for all $i=1,\dots,r$, or $y_i=0$ for all $i=1,\dots,r$.  We are going to show that it is possible to change the basis so that $\alpha_i=-1$ for $i=1,\dots,r-1$ and $x_r+y_r=-2$. 
Suppose that the old basis is $\beta_1=(v_1,w_1,\dots,v_r,w_r,v_{r+1})$. We consider two cases.

{\bf Case 1:} $x_r=0$. Then $y_r\ne 0$ and the new basis is:
\begin{align*}
&v'_{i}=(-1)^{r-i}\alpha_i\cdots\alpha_{r-1}v_i,\ 
w'_{i}=(-1)^{r-i}\alpha_i\cdots\alpha_{r-1}w_i,\ i=1,\dots,r-1,\\
&v'_{r}=v_r,\  w'_{r}=w_r,\quad v'_{r+1}=-\frac{y_r}{2}v_{r+1}.
\end{align*}
In the new basis we have:
\[
D_r=\Psi_1(t_{\delta_{2r+1}}),\quad
D_i=C_i+x_i'\left(E_{2r+1,2i}-E_{2r+1,2i+2}\right),\]
for $i=1,\dots,r-1$.

{\bf Case 2:} $y_r=0$. Then $x_r\ne 0$ and the new basis is:
\begin{align*}
&v'_{i}=(-1)^{r-i+1}\alpha_i\cdots\alpha_{r-1}\frac{x_r}{2}v_i,\ 
w'_{i}=(-1)^{r-i+1}\alpha_i\cdots\alpha_{r-1}\frac{x_r}{2}w_i,\\ &i=1,\dots,r-1,\, 
v'_{r}=-\frac{x_r}{2}v_r,\  w'_{r}=-\frac{x_r}{2}w_r,\quad v'_{r+1}=v_{r+1}.
\end{align*}
In the new basis we have:
\[
D_r=\Psi_2(t_{\delta_{2r+1}}),\quad
D_i=C_i+x_i'\left(E_{2i-1,2r+1}-E_{2i+1,2r+1}\right),\]
for $i=1,\dots,r-1$.

Since $U_{2r+1}$ commutes with $A_i$ and $B_i$ for $1\le i\le r-1$, thus, by  Lemma \ref{diag}, \[U_{2r+1}=\diag\left(\lambda_1I_2,\lambda_2I_2,\dots,\lambda_{r-1}I_2,X\right),\]
for some $3\times 3$ matrix $X$. The relations $A_rU_{2r+1}=U_{2r+1}A_r$ (R8) and $D_rU_{2r+1}D_r=U_{2r+1}$ (R12) imply that $X$ has the form
\[X=\begin{pmatrix}\lambda_r&\alpha&\lambda_r\\0&\lambda_r&0\\0&\beta&-\lambda_r\end{pmatrix}
\quad\mathrm{or}\quad
X=\begin{pmatrix}\lambda_r&\alpha&\beta\\0&\lambda_r&0\\0&\lambda_r&-\lambda_r\end{pmatrix}\]
respectively in case 1 and case 2. For $1\le i\le r-1$, by the relation (R6) we have $D_iU_{2r+1}-U_{2r+1}D_i=0$. By solving this equation we obtain 
$\lambda_i=\lambda_{i+1}$ and $x'_i=0$, hence $D_i=C_i$. We also see that $U_{2r+1}$ has two eigenvalues $\lambda_r$,$-\lambda_r$ with $\#\lambda_r=2r$. Since $U_{2r+1}$ is conjugate to $U_{2r+1}^{-1}$ we have $\lambda_r\in\{-1,1\}$ and by multiplying $f$ by $(-1)^\ab$ if necessary, we may assume $\lambda_r=1$. 

By the relation (R11) we have
\begin{align*}
&U_{2r}=(B_rC_r)^{-1}U_{2r+1}^{-1}(B_rC_r),\\  
&U_{2r-1}=(B_rC_rC_{r-1}B_r)^{-1}U_{2r+1}(B_rC_rC_{r-1}B_r).
\end{align*}
Similarly as in the proof for odd $g$, by solving
$U_{2r+1}U_{2r-1}-U_{2r-1}U_{2r+1}=0$ we obtain $\beta=-2\alpha$, and then by solving $U_{2r+1}U_{2r}U_{2r+1}-U_{2r}U_{2r+1}U_{2r}=0$ we obtain $\alpha=-1$ in the case 1, or $\alpha=1$ in the case 2. Hence $U_{2r+1}=\Psi_1(u_{2r+1})$
in the case 1, or $U_{2r+1}=\Psi_2(u_{2r+1})$
in the case 2. By Theorem \ref{gener}, $f$ is equal to $\Psi_1$ in the case 1, and equal to $\Psi_2$ in the case 2, on generators of $\M(N)$.
\end{proof} 

\section{Homomorphisms from $\M(N_{8})$ to $\GL(7,\C)$}
The aim of this section is to prove Theorem \ref{MN8toGL7}. First we have to define the epimorphism $\epsilon\colon\M(N_{2r+2})\to\Sp(2r,\Z_2)$. 

Fix $r\ge 1$ and set $V=H_1(N_{2r+2},\Z_2)$. $V$ is a vector space over $\Z_2$ of dimension $2r+2$ with basis $\overline{x_i}=[\xi_i]_2$ for $1\le i\le 2r+2$,
where $[\xi_i]_2$ denotes the mod 2 homology class of the curve $\xi_i$.
The mod 2 intersection pairing is the symmetric bilinear form on $V$ satisfying 
$\lr{\overline{x_i},\overline{x_j}}_2=\delta_{ij}$. We define another basis for $V$. For $1\le i\le r$ we set
\begin{align*}
&v_i=[\varepsilon_i]_2=\overline{x_1}+\cdots+\overline{x_{2i}},
\quad w_i=[\delta_{2i}]_2=\overline{x_{2i}}+\overline{x_{2i+1}},\\
&c=\overline{x_{2r+2}},\quad d=\overline{x_1}+\cdots+\overline{x_{2r+2}}.
\end{align*}
Let $\mathrm{Iso}(V)$ denote the group of automorphisms of $V$ preserving $\lr{\cdot,\cdot}_2$.
\begin{lemma}\label{semidir}
The group $\mathrm{Iso}(V)$ is isomorphic to a  semi-direct product 
$\Sp(2r,\Z_2)\ltimes\Z_2^{2r+1}$.
\end{lemma}
\begin{proof}
It is easy to check that $d$ is the unique vector of $V$ satisfying
$\lr{x,d}_2=\lr{x,x}_2$ for all $x\in V$, which implies that $d$ is fixed by all elements of $\mathrm{Iso}(V)$.

Let $W=\mathrm{span}\{v_i,w_i\,|\,i=1,\dots,r\}$ 
and observe that the restriction of $\lr{\cdot,\cdot}_2$ to $W$ is nondegenerate and $\lr{x,x}_2=0$ for $x\in W$, hence it is a symplectic form on $W$. For $R\in\Sp(W)$ we define $A_R\in\mathrm{Iso}(V)$ as
\[A_R(d)=d,\quad A_R(c)=c,\quad A_R(x)=R(x)\quad\textrm{for\ }x\in W.\] It is easy to check that $W=\{x\in V\,|\,\lr{x,d}_2=\lr{x,c}_2=0\}$. It follows that if $L\in\mathrm{Iso}(V)$ fixes $c$, then since $L(d)=d$, $L$ preserves $W$, and hence $L=A_R$ for some $R\in\Sp(W)$. Thus the mapping  $R\mapsto A_R$ defines an isomorphism $\Sp(W)\to\mathrm{Stab}_{\mathrm{Iso}(V)}(c)$.

For $x\in\mathbb{Z}_2$ and $z\in W$ we define $B_{x,z}\in\mathrm{Iso}(V)$ as
\[B_{x,z}(d)=d,\quad B_{x,z}(c)=c+xd+z,\quad B_{x,z}(w)=w+\lr{w,z}_2d\quad\textrm{for\ }w\in W.\]
Let 
\[N=\{B_{x,z}\,|\,x\in\mathbb{Z}_2, z\in W\}.\]
This is a subgroup of $\mathrm{Iso}(V)$ with the group law
\[B_{x_1,z_1}B_{x_2,z_2}=B_{x_1+x_2+\lr{z_1,z_2}_2,z_1+z_2}.\]
It follows that $N$ is abelian and $B_{x,z}^2=1$ for all
$x, z$. Thus $N$ is isomorphic to $\Z_2^{2r+1}$.

Let $L\in\mathrm{Iso}(V)$ be arbitrary. 
Since $\lr{L(c),d}=\lr{L(c),L(d)}=\lr{c,d}=1$, thus
$L(c)=c+xd+z$ for some $x\in\mathbb{Z}_2$, $z\in W$. It follows that
$B_{x,z}^{-1}L\in\mathrm{Stab}_{\mathrm{Iso}(V)}(c)$ and hence
$L=B_{x,z}A_R$ for some $R\in\Sp(W)$. This decomposition is clearly unique, and
since $A_RB_{x,z}A_R^{-1}=B_{x,R(z)},$ thus $N$ is normal in $\mathrm{Iso}(V)$ and
$\mathrm{Iso}(V)=N\rtimes\mathrm{Stab}_{\mathrm{Iso}(V)}(c)$.
\end{proof} 

\begin{lemma}\label{epionSp}
For $r\ge 2$ there is an epimorphism
\[\epsilon\colon\M(N_{2r+2})\to\Sp(2r,\Z_2),\]
whose kernel is normally generated by $t_{\delta_{2r+1}}u_{2r+1}$ and $t_{\delta_{2r+1}}t_{\varepsilon_r}^{-1}$.
\end{lemma}
\begin{proof}
Let $\M=\M(N_{2r+2})$.
The action of $\M$ on $V=H_1(N_{2r+2},\Z_2)$ induces a homomorphism
$\rho\colon\M\to\mathrm{Iso}(V)$, which was proved to be surjective in \cite{GP} and \cite{McCP}, and whose kernel is the normal closure of $t_{\delta_{2r+1}}u_{2r+1}$ by \cite{SzepGD}. By Lemma \ref{semidir}, there exists a normal subgroup $N$ of $\mathrm{Iso}(V)$, such that $\mathrm{Iso}(V)/N$ is isomorphic to $\Sp(2r,\Z_2)$. We define $\epsilon$ to be the composition of $\rho$ with the canonical projection $\mathrm{Iso}(V)\to\mathrm{Iso}(V)/N$.

Let $K$ be the normal closure of $t_{\delta_{2r+1}}u_{2r+1}$ and $t_{\delta_{2r+1}}t_{\varepsilon_r}^{-1}$ in $\M$. We claim that $K\subseteq\ker\epsilon$. We have $t_{\delta_{2r+1}}u_{2r+1}\in\ker\rho\subset\ker\epsilon$. 
For $x\in V$ we have $\rho(t_{\varepsilon_r})(x)=x+\lr{v_r,x}_2v_r$ and
$\rho(t_{\delta_{2r+1}})(x)=x+\lr{[\delta_{2r+1}]_2,x}[\delta_{2r+1}]_2$.
Since $[\delta_{2r+1}]_2=v_r+d$, it is not difficult to check that
$\rho(t_{\delta_{2r+1}})=B_{1,v_r}\circ\rho(t_{\varepsilon_r})$, which gives
$\rho(t_{\delta_{2r+1}}t_{\varepsilon_r}^{-1})\in N$ and $t_{\delta_{2r+1}}t_{\varepsilon_r}^{-1}\in\ker\epsilon$. It follows that there is an induced epimorphism \[\epsilon'\colon\M/K\to\mathrm{Iso}(V)/N\cong\Sp(2r,\Z_2).\]
To prove that $\epsilon'$ is an isomorphism, it suffices to show
$[\M:K]\le|\Sp(2r,\Z_2)|$. We are going to prove the last inequality by exhibiting an epimorphism $\Sp(2r,\Z_2)\to\M/K$.

Observe that the map $\eta\colon\M(S')\to\M/K$ defined to be the composition of  $\iota\colon\M(S')\to\M$ from Corollary \ref{HomSN} with the canonical projection $\pi\colon\M\to\M/K$ is surjective, because $\M$ is generated by twists about curves on $P(S')$ and $t_{\delta_{2r+1}}u_{2r+1}$ by Theorem \ref{gener}. Gluing a disc along  the boundary component of $S'$ bounding a pair of pants with $\alpha_r$ and $\gamma_r$ induces an epimomorphism $\M(S')\to\M(S_{r,1})$ whose  kernel is normally generated by $t_{\gamma_r}t_{\alpha_r}^{-1}$ (see \cite[Proposition 3.8]{KorkSymp}). Since
$\iota(t_{\gamma_r}t_{\alpha_r}^{-1})=t_{\delta_{2r+1}}t_{\varepsilon_r}^{-1}\in K$, it follows that we have an induced epimorphism 
$\eta'\colon\M(S_{r,1})\to\M/K$. There is an epimorphism $\M(S_{r,1})\to\Sp(2r,\Z_2)$ induced by the action of $\M(S_{r,1})$ on $H_1(S_{r,1},\Z_2)$, whose kernel is normally generated by $t_{\alpha_1}^2$ (see \cite[Theorem 5.7]{BGP}, here we are using the assumption $r\ge 2$). By applying Lemma \ref{tsq} (with $i=r$, $j=2r+1$) to $\pi\colon\M\to\M/K$, we have 
$\eta'(t_{\alpha_1}^2)=\pi(t^2_{\delta_1})=1$. It follows that there is an induced epimorphism
$\eta''\colon\Sp(2r,\Z_2)\to\M/K$.  
\begin{displaymath}
\xymatrix{
\M(S') \ar[r]^\iota \ar[d] & \M \ar[r]^\pi & \M/K\\
\M(S_{r,1}) \ar[d] \ar[urr]^{\eta'} & & \\
\Sp(2r,\Z_2) \ar[uurr]^{\eta''} & &
}
\end{displaymath}
The existence of $\eta''$ proves that $\epsilon'$ is an isomorphism and $K=\ker\epsilon$.\end{proof}
\begin{lemma}\label{MN8factSp}
Suppose that $f\colon\M(N_8)\to\GL(7,\C)$ is a homomorphism, such that
$f(t_{\delta_1})$ has order $2$. Then $f$ or $(-1)^\ab f$ factors through the epimorphism
$\epsilon\colon\M(N_8)\to\Sp(6,\Z_2)$.
\end{lemma}
\begin{proof}
Let $H$ be the normal closure of $t_{\delta_1}^2$ in $\M=\M(N_8)$ and $G=\M/H$.
Since $H\subseteq\ker f$,  
we have a homomorphism $f'\colon G\to\GL(7,\C)$ such that $f=f'\circ\pi$, where   $\pi\colon\M\to G$ is the canonical projection.  There is a homomorphism
$\rho\colon\mathfrak{S}_8\to G$, defined as 
$\rho(\sigma_i)=\pi(t_{\delta_i})$, where $\sigma_i=(i,i+1)$, for $1\le i\le 7$.
Let 
$\phi\colon\mathfrak{S}_8\to\GL(7,\C)$ be the composition $\phi=f'\circ \rho$. 
If $\phi$ is reducible, then $\im(\phi)$ is abelian by Lemma \ref{repsym},  $f(t_{\delta_1})=\phi(\sigma_1)=\phi(\sigma_2)=f(t_{\delta_2})$, and $\im(f)$  is also abelian by Lemma \ref{com_norm}, which implies $f(t_{\delta_1})=1$ by Theorem \ref{abNg}, a contradiction. Hence $\phi$ is irreducible and since $\det f(t_{\delta_1})=1$ (by Theorem \ref{abNg}), $\phi$ is the tensor product of the standard  and  sign representations (by Lemma \ref{repsym}). 
For $1\le i\le 7$ set $L_i=f(t_{\delta_i})=\phi(\sigma_i)$. With respect to some basis   
$(v_1,\dots,v_7)$ 
we have
\begin{align*}
L_1=\diag\left(A,-I_5\right),\quad L_7=\diag\left(-I_5,B\right),\quad
L_i=\diag\left(-I_{i-2},C,-I_{6-i}\right)
\end{align*}
for $2\le i\le 6$, where
\[A=\begin{pmatrix}1&-1\\0&-1\end{pmatrix},\quad
B=\begin{pmatrix}-1&0\\-1&1\end{pmatrix},\quad
C=\begin{pmatrix}-1&0&0\\-1&1&-1\\0&0&-1\end{pmatrix}.\quad
\]
Let $M$ be the matrix of $f(\varepsilon_3)$. Since $M$ commutes with $L_i$ for $i\ne 6$ (R5), it preserves $E(L_i,1)=\mathrm{span}\{v_i\}$. Hence
$M(v_i)=x_iv_i$ for $i\ne 6$ and $M(v_6)=y_1v_1+\cdots+y_7v_7$, for some complex numbers $x_i, y_j$.
By solving the equations $ML_i=L_iM$ for $1\le i\le 5$ and $i=7$ we obtain 
\[x_i=x_1,\ y_i=iy_1\ \textrm{\ for\ }1\le i\le 5,\quad y_6=x_1+6y_1,\  x_7=y_6-2y_7.\]
Since $M$ and $L_i$ are conjugate, they have the same eigenvalues, which gives  $x_1=-1$ and $y_6=-x_7$. If $y_6=1$, then $y_1=1/3$, $y_7=1$, which contradicts the braid relation $ML_6M=L_6ML_6$ (R5). Hence $y_6=-1$, $y_1=0$, $y_7=-1$, which means $M=L_7$. 

For $i=1,\dots 7$ let $U_i$ be the matrix of $f(u_i)$. Since $U_7$ commutes with
$L_j$ for $1\le j\le 5$ (R6) and with $M=L_7$ (R8), we obtain, as above, that
\begin{align*}
&U_7(v_i)=xv_i\quad\textrm{for\ }1\le i\le 5,\\
&U_7(v_6)=y(v_1+2v_2+3v_3+4v_4+5v_5)+(x+6y)v_6+zv_7\\
&U_7(v_7)=(x+6y-2z)v_7
\end{align*}
for some complex numbers $x,y,z$. Since $U_7$ is conjugate to its inverse, and $x$ is an eigenvalue of multiplicity at least $5$, thus $x=\pm 1$, and by multiplying $f$ by $(-1)^\ab$ if necessary, we may assume $x=-1$. By (R11) we have $U_5=(L_6L_7L_5L_6)^{-1}U_7(L_6L_7L_5L_6)$  and by solving $U_5U_7=U_7U_5$ we obtain $y=0$. Since
$\det U_7=\pm 1$, either $-1-2z=1$ or $-1-2z=-1$. In the latter case we have $U_7=-I$, and since $U_6$ is conjugate to $U_7$, thus $U_6=-I$, and the relation $L_6U_7U_6=U_7U_6L_7$ (R10) gives $L_6=L_7$, a contradiction. Hence $z=-1$ and $U_7=L_7$.

We have $M=U_7=L_7$ and since $L_7^2=I$, thus $\{t_{\delta_7}t_{\varepsilon_3}^{-1}, t_{\delta_7}u_7\}\subset\ker f$, which implies, by Lemma \ref{epionSp}, that $f$ factors through $\epsilon$.
\end{proof}

\begin{proof}[Proof of Theorem \ref{MN8toGL7}]
Suppose that $f\colon\M(N_8)\to\GL(7,\C)$ is a homomorphism, such that
$\im(f)$ is not abelian. By Lemma \ref{dim_eigen}, $L=f(t_{\delta_1})$.
has an eigenvalue $\lambda$ such that $\dim E(L,\lambda)=6$. 
Since $L$ is conjugate to $L^{-1}$ we have $\lambda^2=1$.
Suppose that $\lambda=-1$. Then since $\det L=1$ we have $\#\lambda=6$, and there is another eigenvalue $\mu=1$. 
It follows that $L$ has order $2$ and the case (2) holds by Lemma \ref{MN8factSp}.
If $\lambda=1$ then it must be the unique eigenvalue, and the case (3) holds by Lemma \ref{AB} and the proof of Theorem \ref{MNtoGLg-1} for even $g$.
\end{proof}
\begin{rem}
Suppose that $G$ is a finite quotient of $\M(N_g)$ for $g\ge 7$, $g\ne 8$, and
$f\colon G\to\GL(g-1,\C)$ is a homomorphism. Then, by Theorem \ref{MNtoGLg-1}, $\im(f)$ is abelian, and if $G$ is perfect, then $f$ must be trivial. For example, by Lemma \ref{epionSp}, for $r\ge 4$, the only homomorphism from $\Sp(2r,\Z_2)$ to $\GL(2r+1,\C)$ is the trivial one. 
\end{rem}
%
%

%

\end{document}